\documentclass[12pt,a4 paper]{article}
\usepackage{epsfig}
\usepackage{amsmath}
 \usepackage[pdftex]{hyperref}
\usepackage{amsfonts}
\usepackage{amsmath,amscd}
\usepackage{fullpage}
\usepackage{amssymb}
\usepackage{amsthm}
\usepackage{mathrsfs}
\usepackage{graphicx}
\usepackage{xypic}
\newtheorem{definition}{Definition}[section]
\newtheorem{theorem}[definition]{Theorem}
\newtheorem{remark}[definition]{Remark}
\newtheorem{final Remarks}[definition]{Final Remarks}
\newtheorem{lemma}[definition]{Lemma}

\newtheorem{example}[definition]{Example}

\numberwithin{equation}{section}
\begin{document}
\title{The classification of smooth structures on a homotopy complex projective space}
\vspace{2cm}
\author{Ramesh Kasilingam
\\{Theoretical Statistics and Mathematics Unit}\\{Indian Statistical Institute, Kolkata}\\{India.}\\e-mail : {rameshkasilingam.iitb@gmail.com}}
\date{}
\maketitle
\begin{abstract}
We classify, up to diffeomorphism, all closed smooth manifolds homeomorphic to the complex projective $n$-space $\mathbb{C}\textbf{P}^n$, where $n=3$ and $4$. Let $M^{2n}$ be a closed smooth $2n$-manifold homotopy equivalent to $\mathbb{C}\textbf{P}^n$. We show that, up to diffeomorphism, $M^{6}$ has a unique differentiable structure and $M^{8}$ has at most two distinct differentiable structures. We also show that, up to concordance, there exist at least two distinct differentiable structures on a finite sheeted cover $N^{2n}$ of $\mathbb{C}\textbf{P}^n$ for $n=4, 7$ or $8$ and six distinct differentiable structures on $N^{10}$.
\end{abstract}
\paragraph{Classification.}
57R55; 57R50.
\section{Introduction}
A piecewise linear homotopy complex projective space $M^{2n}$ is a
closed PL $2n$-manifold homotopy equivalent to the complex projective space $\mathbb{C}\textbf{P}^n$. In \cite{Sul67}, Sullivan gave a complete enumeration of the set of PL isomorphism classes of these manifolds as a consequence of his Characteristic Variety theorem and his analysis of the homotopy type of $G/PL$. He also proved that the group of concordance classes of smoothing of $\mathbb{C}\textbf{P}^n$ is in one-to-one correspondence with the set of
$c$-oriented diffeomorphism classes of smooth manifolds homeomorphic (or PL-homeomorphic) to $\mathbb{C}\textbf{P}^n$, where $c$ is the generator of $H^2(\mathbb{C}\textbf{P}^n;\mathbb{Z})$.\\
\indent In section 2, we classify up to diffeomorphism all closed smooth manifolds homeomorphic to $\mathbb{C}\textbf{P}^n$, where $n=3$ and $4$.\\
Let $M^{2n}$ be a closed smooth $2n$-manifold homotopy equivalent to $\mathbb{C}\textbf{P}^n$. The surgery theory tells us that there are infinitely many diffeomorphism types in the family of closed smooth manifolds homotopy equivalent to $\mathbb{C}\textbf{P}^n$ when $n\geq 3$. In the second section, we also show that if $N$ is a closed smooth manifold homeomorphic to $M^{2n}$, where $n=3$ or $4$, there is a homotopy sphere $\Sigma\in \Theta_{2n}$ such that $N$ is diffeomorphic to $M\#\Sigma$. In particular, up to diffeomorphism, $M^{6}$ has a unique differentiable structure and $M^{8}$ has at most two distinct differentiable structures.\\
In section 3, we prove that if $N^{2n}$ is a finite sheeted cover of $\mathbb{C}\textbf{P}^n$, then up to concordance, there exist at least $|\Theta_{2n}|$ distinct differentiable structures on $N^{2n}$, namely $\{ [N^{2n}\#\Sigma]~~|~~ \Sigma\in \Theta_{2n}\}$, where $n=4, 5, 7$ or $8$ and $|\Theta_{2n}|$ is the order of $\Theta_{2n}$. 
\section{Smooth Structures on Complex Projective Spaces}
We recall some terminology from \cite{KM63}:
\begin{definition}\rm{
\begin{itemize}
\item[(a)] A homotopy $m$-sphere $\Sigma^m$ is an oriented smooth closed manifold homotopy equivalent to the standard unit sphere $\mathbb{S}^m$ in $\mathbb{R}^{m+1}$.
\item[(b)]A homotopy $m$-sphere $\Sigma^m$ is said to be exotic if it is not diffeomorphic to $\mathbb{S}^m$.
\item[(c)] Two homotopy $m$-spheres $\Sigma^{m}_{1}$ and $\Sigma^{m}_{2}$ are said to be equivalent if there exists an orientation preserving diffeomorphism $f:\Sigma^{m}_{1}\to \Sigma^{m}_{2}$.
\end{itemize}
The set of equivalence classes of homotopy $m$-spheres is denoted by $\Theta_m$. The equivalence class of $\Sigma^m$ is denoted by [$\Sigma^m$]. When $m\geq 5$, $\Theta_m$ forms an abelian group with group operation given by connected sum $\#$ and the zero element represented by the equivalence class  of $\mathbb{S}^m$. M. Kervaire and J. Milnor \cite{KM63} showed that each  $\Theta_m$ is a finite group; in particular,  $\Theta_{8}$ and  $\Theta_{16}$ are cyclic groups of order $2$.}
\end{definition}
\begin{definition}\rm
Let $M$ be a topological manifold. Let $(N,f)$ be a pair consisting of a smooth manifold $N$ together with a homeomorphism $f:N\to M$. Two such pairs $(N_{1},f_{1})$ and $(N_{2},f_{2})$ are concordant provided there exists a diffeomorphism $g:N_{1}\to N_{2}$ such that the composition $f_{2}\circ g$ is topologically concordant to $f_{1}$, i.e., there exists a homeomorphism $F: N_{1}\times [0,1]\to M\times [0,1]$ such that $F_{|N_{1}\times 0}=f_{1}$ and $F_{|N_{1}\times 1}=f_{2}\circ g$. The set of all such concordance classes is denoted by $\mathcal{C}(M)$.\\
\indent Start by noting that there is a homeomorphism $h: M^n\#\Sigma^n \to M^n$ $(n\geq5)$ which is the inclusion map outside of homotopy sphere $\Sigma^n$ and well defined up to topological concordance. We will denote the class in $\mathcal{C}(M)$ of $(M^n\#\Sigma^n, h)$ by $[M^n\#\Sigma^n]$. (Note that $[M^n\#\mathbb{S}^n]$ is the class of $(M^n, Id)$.)
\end{definition}
\begin{theorem}\label{con.oct}
\begin{itemize}
 \item[(i)]$\mathcal{C}(\mathbb{C}\textbf{P}^3)=0$.
 \item[(ii)]$\mathcal{C}(\mathbb{C}\textbf{P}^4)=\{[\mathbb{C}\textbf{P}^4],[\mathbb{C}\textbf{P}^4\#\Sigma^{8}]\}\cong \mathbb{Z}_2$. 
 \end{itemize}
\end{theorem}
\begin{proof}
(i): Consider the following Puppe's exact sequence for the inclusion $i:\mathbb{C}\mathbb{P}^{n-1}\hookrightarrow \mathbb{C}\mathbb{P}^{n}$ along $Top/O$:
\begin{equation}\label{longG}
....\longrightarrow [S\mathbb{C}\mathbb{P}^{n-1}, Top/O]\stackrel{(S(g))^{*}}{\longrightarrow}[\mathbb{S}^{2n}, Top/O]\stackrel{f^{*}_{\mathbb{C}\mathbb{P}^{n}}}{\longrightarrow}[\mathbb{C}\mathbb{P}^{n}, Top/O]\stackrel{i^{*}}{\longrightarrow}[\mathbb{C}\mathbb{P}^{n-1}, Top/O],
\end{equation}
where $S(g)$ is the suspension of the map $g:\mathbb{S}^{2n-1}\to \mathbb{C}\mathbb{P}^{n-1}$.
If $n=2$ or $3$ in the above exact sequence (\ref{longG}), we can prove that $[\mathbb{C}\mathbb{P}^{n}, Top/O]=0$. Now by using the identifications $\mathcal{C}(\mathbb{C}\textbf{P}^{3})=[\mathbb{C}\textbf{P}^{3}, Top/O]$ given by \cite[pp. 194-196]{KS77}, $\mathcal{C}(\mathbb{C}\textbf{P}^{3})=0$. This proves (i).\\
(ii): Now consider the case $n=4$ in the above exact sequence (\ref{longG}), we have that $f^{*}_{\mathbb{C}\textbf{P}^{4}}:[\mathbb{S}^8, Top/O]\cong \Theta_8\mapsto [\mathbb{C}\mathbb{P}^{4}, Top/O]$ is surjective. Then by using \cite[Lemma 3.17]{FJ94}, $f^{*}_{\mathbb{C}\textbf{P}^{4}}$ is an isomorphism. Hence $\mathcal{C}(\mathbb{C}\textbf{P}^4)=\{[\mathbb{C}\textbf{P}^4],[\mathbb{C}\textbf{P}^4\#\Sigma^{8}]\}\cong \mathbb{Z}_2$. This proves (ii).
\end{proof}
\begin{definition}\label{homo.iner}\rm{
Let $M^m$ be a closed smooth, oriented $m$-dimensional manifold. The inertia group $I(M)\subset \Theta_{m}$ is defined as the set of $\Sigma \in \Theta_{m}$ for which there exists an orientation preserving diffeomorphism $\phi :M\to M\#\Sigma$.\\
Define the concordance inertia group $I_c(M)$ to be the set of all $\Sigma\in I(M)$ such that $M\#\Sigma$ is concordant to $M$.}
\end{definition}
\begin{theorem}{\rm \cite[Theorem 4.2]{Kas14}}\label{first}
For $n \geq 1$, $I_c(\mathbb{C}\mathbb{P}^{n})=I(\mathbb{C}\mathbb{P}^{n}).$
\end{theorem}
\begin{remark}\label{hoinva}\rm
\indent
\begin{itemize}
\item [(1)] By Theorem \ref{con.oct} and Theorem \ref{first}, $I_c(\mathbb{C}\textbf{P}^n)=0=I(\mathbb{C}\textbf{P}^n)$, where $n=3$ and $4$.
\item [(2)] By Kirby and Siebenmann identifications \cite[pp. 194-196]{KS77}, the group $\mathcal{C}(M)$ is a homotopy invariant.
\end{itemize}
\end{remark}
\begin{theorem}\label{class}
Let $M^{2n}$ be a closed smooth $2n$-manifold homotopy equivalent to $\mathbb{C}\textbf{P}^n$.
\begin{itemize}
\item[(i)] For $n=3$, $M^{2n}$ has a unique differentiable structure up to diffeomorphism.
\item[(ii)] For $n=4$, $M^{2n}$ has at most two distinct differentiable structures up to diffeomorphism.
\end{itemize}
Moreover, if $N$ is a closed smooth manifold homeomorphic to $M^{2n}$, where $n=3$ or $4$, there is a homotopy sphere $\Sigma\in \Theta_{2n}$ such that $N$ is diffeomorphic to $M\#\Sigma$.
\end{theorem}
\begin{proof}
Let $N$ be a closed smooth manifold homeomorphic to $M$ and let $f:N\to M$ be a homeomorphism. Then $(N,f)$ represents an element in $\mathcal{C}(M)$. By Theorem \ref{con.oct} and Remark \ref{hoinva}(2), there is a homotopy sphere $\Sigma\in \Theta_{2n}$ such that $N$ is concordant to $(M\#\Sigma, Id)$. This implies that $N$ is diffeomorphic to $M\#\Sigma$. This proves the theorem.
\end{proof}
\begin{remark}\rm
Since $\Theta_8\cong \mathbb{Z}_2$ and $I(\mathbb{C}\textbf{P}^4)=0$, by Theorem \ref{class}, $\mathbb{C}\textbf{P}^4$ has exactly two distinct differentiable structures up to diffeomorphism.
\end{remark}
\section{Tangential types of Complex Projective Spaces}
\begin{definition}\rm
Let $M^n$ and $N^n$ be closed oriented smooth $n$-manifolds. We call $M$ a tangential type of $N$ if there is a smooth map $f:M\to N$ such $f^*(TN)=TM$, where $TM$ is the tangent bundle of $M$.
\end{definition}
\begin{example}\rm
\indent
\begin{itemize}
\item[(i)] Every finite sheeted cover of $\mathbb{C}\textbf{P}^n$ is a tangential type of $\mathbb{C}\textbf{P}^n$.
\item[(ii)] Since Borel \cite{Bor63} has constructed closed complex hyperbolic manifolds in every complex dimension $m\geq 1$, by \cite[Theorem 5.1]{Oku01}, there exists a closed  complex hyperbolic manifold $M^{2n}$ which is a tangential type of $\mathbb{C}\textbf{P}^n$.
\end{itemize}
\end{example}
\begin{lemma}\label{okun}{\rm \cite[Lemma 2.5]{Oku02}}
Let $M^{2n}$ be a tangential type of $\mathbb{C}\textbf{P}^n$ and assume $n\geq4$. Let $\Sigma_{1}$ and  $\Sigma_{2}$ be homotopy $2n$-spheres. Suppose that $M^{2n}\#\Sigma_{1}$ is concordant to $M^{2n}\#\Sigma_{2}$, then $\mathbb{C}\textbf{P}^n\#\Sigma_{1}$ is concordant to $\mathbb{C}\textbf{P}^n\#\Sigma_{2}$.
\end{lemma}
\begin{theorem}\label{kawaku}{\rm \cite{Kaw68}}
For $n\leq 8$, $I(\mathbb{C}\textbf{P}^n)=0$.
\end{theorem}
\begin{theorem}\label{tang}
Let $M^{2n}$ be a tangential type of $\mathbb{C}\textbf{P}^n$. Then 
\begin{itemize}
\item[(i)] For $n\leq 8$, the concordance inertia group $I_c(M^{2n})=0$.
\item[(ii)] For $n=4k+1$, where $k\geq 1$, $$I_c(M^{2n})\neq \Theta_{2n}.$$ Moreover, if $M^{2n}$ is simply connected, then $$I(M^{2n})\neq \Theta_{2n}.$$
\end{itemize}
\end{theorem}
\begin{proof}
(i): By Theorem \ref{kawaku}, for $n\leq 8$, $I(\mathbb{C}\textbf{P}^n)=0$ and hence $I_c(\mathbb{C}\textbf{P}^n)=0$. Now by Theorem \ref{okun}, $I_c(M^{2n})=0$. This proves (i).\\
(ii): By \cite[Proposition 9.2]{Kaw69}, for $n=4k+1$,  there exists a homotopy $2n$-sphere $\Sigma$ not bounding spin-manifold such that $\mathbb{C}\textbf{P}^n\#\Sigma$ is not concordant to $\mathbb{C}\textbf{P}^n$. Hence by Theorem \ref{okun}, $$I_c(M^{2n})\neq \Theta_{2n}.$$ Moreover, $\mathbb{C}\textbf{P}^n$ is a spin manifold and hence the Stiefel-Whitney class $w_i(\mathbb{C}\textbf{P}^n)=0$, where $i=1$ and $2$. Since $M^{2n}$ is a tangential type of $\mathbb{C}\textbf{P}^n$, there is a smooth map $f:M^{2n}\to \mathbb{C}\textbf{P}^n$ such that $f^*(T\mathbb{C}\textbf{P}^n)=TM^{2n}$. This implies that $w_i(M^{2n})=f^*(w_i(\mathbb{C}\textbf{P}^n))=0$. So, $M^{2n}$ is a spin manifold. If $M^{2n}$ is  simply connected, then by \cite[Lemma 9.1]{Kaw69}, $\Sigma\notin I(M^{2n})$ and hence $$I(M^{2n})\neq \Theta_{2n}.$$ This proves the theorem.
\end{proof}
\begin{remark}\rm
Let $M^{2n}$ be a tangential type of $\mathbb{C}\textbf{P}^n$.  By Theorem \ref{tang}, up to concordance, there exist at least $|\Theta_{2n}|$ distinct differentiable structures, namely $\{ [M^{2n}\#\Sigma]~~|~~ \Sigma\in \Theta_{2n}\}$, where $n=4, 5, 7$ or $8$ and $|\Theta_{2n}|$ is the order of $\Theta_{2n}$.
\end{remark}

\end{document}